\documentclass[12pt]{amsart}
\usepackage{color}
\usepackage{enumerate}
\usepackage{graphics}
\usepackage[colorlinks, urlcolor=blue]{hyperref}
\usepackage{mathrsfs}
\usepackage{amssymb}

\newtheorem{thm}{Theorem}
\newtheorem*{theorem*}{Theorem}
\newtheorem*{problem*}{Problem}

\newtheorem{prop}[thm]{Proposition}

\newtheorem{defn}[thm]{Definition}

\usepackage[left=2.5cm,right=2.5cm,top=2cm,bottom=2.5cm]{geometry}

\theoremstyle{remark}

\newtheorem{rem}[thm]{Remark}

\begin{document}

\title[Complete boundedness of multiple operator integrals]{Complete boundedness of multiple operator integrals}

\author[C. Coine]{Cl\'ement Coine}
\email{clement.coine1@gmail.com}

\address{School of Mathematics and Statistics, Central South University, Changsha 410085,
People’s Republic of China}


\maketitle

\begin{abstract}
In this paper, we characterize the multiple operator integrals mappings which are bounded on the Haagerup tensor product of spaces of compact operators. We show that such maps are automatically completely bounded and prove that this is equivalent to a certain factorization property of the symbol associated to the operator integral mapping. This generalizes a result by Juschenko-Todorov-Turowska on the boundedness of continuous multilinear Schur multipliers.
\end{abstract}

\bibliographystyle{short}

\section{Introduction}

A family $m=(m_{ij})_{i,j \in \mathbb{N}}$ of complex numbers is called a Schur multiplier if for any matrix $[a_{ij}] \in \mathcal{B}(\ell_2)$, the Schur product $T_m(a) = [m_{ij}a_{ij}]$ is the matrix of an element of $\mathcal{B}(\ell_2)$.
Schur multipliers are an important tool in analysis, and play for instance a fundamental role in Perturbation Theory. See below for more informations and references.

There is a well-known characterization of Schur multipliers due to Grothendieck in terms of factorization of the symbol $m$, see \cite[Theorem 5.1]{PisierBook}. It turns out, using the theory of operator spaces, that bounded Schur multipliers are completely bounded and in that case, the norm of $T_m$ is equal to its complete norm. To this day, it is still unknown whether this is true for Schur multipliers defined on the Schatten classes. We refer to \cite{Caspers} for recent developments regarding this question.

In this paper, we are interested in Schur multipliers in the multilinear setting. Effros and Ruan \cite{ER} introduced a Schur product as a multilinear map $T \colon M_n(\mathbb{C}) \times \cdots \times M_n(\mathbb{C}) \rightarrow M_n(\mathbb{C})$ defined on the product of $n$ copies of $M_n(\mathbb{C})$ and characterized the mappings $T$ which extend to a complete contraction on the Haagerup tensor product $M_n(\mathbb{C}) \overset{h}{\otimes} \cdots \overset{h}{\otimes} M_n(\mathbb{C})$. This result was generalized by Juschenko, Todorov and Turowska in \cite{JTT} where they considered continuous multilinear Schur multipliers. They are defined as follows: let $n\in \mathbb{N}$ and let $(\Omega_1, \mu_1), \ldots, (\Omega_n, \mu_n)$ be $\sigma$-finite measure spaces. Let $\phi \in L^{\infty}(\Omega_1 \times\cdots \times \Omega_n)$. If $K_i \in L^2(\Omega_i\times \Omega_{i+1})$, $1\leq i \leq n-1$, we let $\Lambda(\phi)(K_1,\ldots, K_{n-1})$ to be the Hilbert-Schmidt operator with kernel
$$
\int
\phi(t_1,\ldots,t_n)K_1(t_1,t_2)\ldots K_{n-1}(t_{n-1},t_n) \,\text{d}\mu_2(t_2) \ldots \text{d}\mu_{n-1}(t_{n-1}) \in L^2(\Omega_1 \times \Omega_n).
$$
Identifying $L^2(\Omega_i\times \Omega_j)$ with $\mathcal{S}^2(L^2(\Omega_i), L^2(\Omega_j))$, this defines a multilinear mapping
$$
\Lambda(\phi) \colon \mathcal{S}^2(L^2(\Omega_{n-1}), L^2(\Omega_n)) \times \cdots \times \mathcal{S}^2(L^2(\Omega_1), L^2(\Omega_2)) \rightarrow S^2(L^2(\Omega_1), L^2(\Omega_n)).
$$
Using the notion of multilinear module mappings, the authors proved that if $\Lambda(\phi)$ extends to a bounded map on the Haagerup tensor product $\mathcal{S}^{\infty}(L^2(\Omega_{n-1}), L^2(\Omega_n)) \overset{h}{\otimes} \cdots \overset{h}{\otimes} \mathcal{S}^{\infty}(L^2(\Omega_1), L^2(\Omega_2))$ into $\mathcal{S}^{\infty}(L^2(\Omega_1), L^2(\Omega_n))$, the extension is completely bounded \cite[Lemma  3.3]{JTT}. Using this fact, they characterized the functions $\phi$ which give rise to a (completely) bounded $\Lambda(\phi)$ in terms of the extended Haagerup tensor product $L^{\infty}(\Omega_1) \otimes_{eh} \cdots \otimes_{eh} L^{\infty}(\Omega_n)$,  see \cite[Theorem 3.4]{JTT} and the remark following the theorem. We also refer to \cite{Spronk} for more results on the case $n=2$.

Let $A_1, \ldots, A_n$ be normal operators and let $\lambda_{A_1}, \ldots, \lambda_{A_n}$ be scalar-valued spectral measures associated to these operators, that is, $\lambda_{A_i}$ is a finite measure on the Borel subsets of $\sigma(A_i)$ such that $\lambda_{A_i}$  and $E^{A_i}$, the spectral measure of $A_i$, have the same sets of measure $0$. For $\phi \in L^{\infty}(\lambda_{A_1} \times \cdots \times \lambda_{A_n})$ and $X_1, \ldots, X_{n-1} \in \mathcal{S}^2(\mathcal{H})$, we formally define a multiple operator integral by
\begin{align*}
& \left[\Gamma^{A_1,  \ldots, A_n}(\phi)\right](X_1, \ldots, X_{n-1})\\
& \ \ \ \ \ \ \ \ = \int_{\sigma(A_1)\times \cdots \times \sigma(A_n)} \phi(s_1, \ldots, s_n)\,
\text{d}E^{A_1}(s_1)\, X_1\, \text{d}E^{A_2}(s_2) \ldots X_{n-1} \, \text{d}E^{A_n}(s_n).
\end{align*}
The theory of double operator integral (case $n=2$) was developed by Birman and Solomyak in a series of three papers \cite{BS1, BS2, BS3} and was then generalized to the case of multiple operator integrals \cite{Pav, Stek}. They play a prominent role in operator theory, especially in perturbation theory where they are a fundamental tool in the study of differentiability of operator functions. See \cite{CLSS, Coine, LMS, Peller2006} where Fr\'echet and G\^ateaux-differentiability of the mapping $f \mapsto f(A)$ are studied in the Schatten norms.

The definition of multiple operator integrals we will use in this paper is the one given in \cite{CLS} and which is based on the construction of Pavlov \cite{Pav}. See \cite{Peller2006,PSS-SSF} for other constructions of multiple operator integrals. The advantages of this definition is that any bounded Borel function is integrable and the property of $w^*$-continuity of the mapping $\phi \mapsto \Gamma^{A_1,  \ldots, A_n}(\phi)$ which allows to prove certain identities by simply checking them for functions with separated variables, see \cite{CLS, CLSS} and the proof of Theorem $\ref{CBMOI}$.

In this paper, we prove that a similar characterization than that of continuous multilinear Schur multipliers \cite{JTT} holds in the setting of multiple operator integrals. Namely, we prove that if a multiple operator integral $\Gamma^{A_1, \ldots, A_n}$ extends to a bounded mapping on the Haagerup tensor product $\mathcal{S}^{\infty}(\mathcal{H}) \overset{h}{\otimes} \cdots \overset{h}{\otimes} \mathcal{S}^{\infty}(\mathcal{H})$ then the extension is completely bounded and that we have such an extension if and only if $\phi$ has the following factorization: there exist separable Hilbert spaces $H_1,  \ldots, H_{n-1}$, $a_1\in L^{\infty}(\lambda_{A_1} ; H_1), a_n \in L^{\infty}(\lambda_{A_n} ; H_{n-1})$ and $a_i\in L_{\sigma}^{\infty}(\lambda_{A_i} ; \mathcal{B}(H_i, H_{i-1})), 2\leq i \leq n-1$
such that
\begin{equation*}
\phi(t_1,\ldots,t_n)= \left\langle a_1(t_1), [a_2(t_2)\ldots a_{n-1}(t_{n-1})](a_n(t_n)) \right\rangle.
\end{equation*}
Our proof rests on several properties of the Haagerup tensor product (Section $\ref{Haagerup}$) and the connection between multiple operator integrals and continuous multilinear Schur multipliers that we will present in Section $\ref{MOI}$.

\section{Preliminaries}

\subsection{Operator spaces and the Haagerup tensor product}\label{Haagerup}

We refer to \cite{PisierCB} and \cite{Ruan} for the theory of operator spaces. If $E \subset \mathcal{B}(H)$ and $F \subset \mathcal{B}(K)$ are two operator spaces, we denote by $CB(E,F)$ the Banach space of completely bounded maps from $E$ into $F$ equipped with the c.b. norm. If $\mathcal{H}$ is a Hilbert space, we will denote by $\mathcal{H}_c = \mathcal{B}(\mathbb{C}, \mathcal{H})$ its column structure.\\
In this subsection, we will recall a few properties of the Haagerup tensor product $E_1 \overset{h}{\otimes} \cdots \overset{h}{\otimes} E_N$ of $N$ operator spaces $E_1, \ldots, E_N$. See \cite[Chapter 5]{PisierCB} for a definition. The first property is the factorization of multilinear maps.

\begin{thm}\label{DualityH}
Let $E_1, \ldots, E_n$ be operator spaces and let $H_0$ and $H_n$ be Hilbert spaces. A linear mapping 
$u \colon E_1 \overset{h}{\otimes} \cdots \overset{h}{\otimes} E_n \rightarrow \mathcal{B}(H_n, H_0)$
is completely bounded if and only if there exist Hilbert spaces $H_1, \ldots, H_{n-1}$ and completely bounded mappings $\phi_i \colon E_i \rightarrow \mathcal{B}(H_i, H_{i-1}), 1\leq i \leq n,$ such that
$$u(x_1 \otimes \cdots \otimes x_n) = \phi_1(x_1) \ldots \phi_n(x_n).$$
In this case we can choose $\phi_i, 1\leq i \leq n$, such that
$$\|u\|_{\text{cb}} = \|\phi_1\|_{\text{cb}} \cdots \|\phi_n\|_{\text{cb}}.$$
\end{thm}

\begin{rem}\label{DualH}
When $H_0=H_n=\mathbb{C}$ we can reformulate as follows: a linear functional $u \colon E_1 \overset{h}{\otimes} \cdots \overset{h}{\otimes} E_n \rightarrow \mathbb{C}$ is bounded (and therefore completely bounded) if and only if there exist Hilbert spaces $H_1, \ldots, H_{n-1}$, $\alpha_1 \colon E_1 \rightarrow (H_c)^*$ linear, $\alpha_i \colon E_i \rightarrow \mathcal{B}(H_i, H_{i-1}), 2\leq i\leq n-1$ and $\alpha_n \colon E_n \rightarrow (H_{n-1})_c$ antilinear such that the $\alpha_j$ are completely bounded and
$$u(x_1, \ldots, x_n) = \left\langle \alpha_1(x_1), [\alpha_2(x_2) \ldots \alpha_{n-1}(x_{n-1})]\alpha_n(x_n) \right\rangle.$$
\end{rem}

Recall that a map $s \colon X \rightarrow Y$ between two Banach spaces is called a quotient map if the injective map $\hat{s} \colon X/\ker(s) \rightarrow Y$ induced by $s$ is a surjective isometry. If $E_1 \subset E_2$ are operator spaces, we equip $E_2/E_1$ with the quotient operator space structure (see e.g. \cite[Section 2.4]{PisierCB}).
When $E$ and $F$ are operator spaces, a quotient map $u \colon E \rightarrow F$ is said to be a \textit{complete metric surjection} if the associated mapping $\hat{u} \colon E/\ker(u) \rightarrow F$ is a completely isometric isomorphism.

\begin{prop}\label{Injectivity} Let $E_1, E_2, F_1, F_2$ be operator spaces.
\begin{enumerate}
\item[(i)] If $q_i \colon E_i \rightarrow F_i$ is completely bounded, then  $q_1 \otimes q_2 \colon E_1 \otimes E_2 \rightarrow F_1 \overset{h}{\otimes} F_2$
defined by $(q_1 \otimes q_2)(e_1 \otimes e_2) = q_1(e_1) \otimes q_2(e_2)$
extends to a completely bounded map
$$q_1 \otimes q_2 \colon E_1 \overset{h}{\otimes} E_2 \rightarrow F_1 \overset{h}{\otimes} F_2.$$
\item[(ii)] If $E_i \subset F_i$ completely isometrically, then $E_1 \overset{h}{\otimes} E_2 \subset F_1 \overset{h}{\otimes} F_2$ completely isometrically.
\item[(iii)] If $q_i \colon E_i \rightarrow F_i$ is a complete metric surjection, then $q_1 \otimes q_2 \colon E_1 \overset{h}{\otimes} E_2 \rightarrow F_1 \overset{h}{\otimes} F_2$ is also one.
\item[(iv)] If $E_i \subset F_i$ are subspaces, let $p_i \colon F_i \rightarrow F_i/E_i$ be the canonical mappings. Then, the induced map $p_1 \otimes p_2 \colon F_1  \overset{h}{\otimes} F_2 \rightarrow F_1/E_1 \overset{h}{\otimes} F_2/E_2$ satisfies
$$\ker (p_1\otimes p_2) = \overline{E_1 \otimes F_2 + F_1 \otimes E_2}.$$
\end{enumerate}
The second property is called the injectivity and the third one the projectivity of the Haagerup tensor product.
\end{prop}

\begin{proof} We refer to \cite[Proposition 9.2.5]{Ruan} for the proof of $(i)$ and to \cite[Corollary 5.7]{PisierCB} for the proof of $(ii)$ and $(iii)$.

Let us prove $(iv)$.
Write $N = \overline{E_1 \otimes F_2 + F_1 \otimes E_2}$. Note that the inclusion $N \subset \ker(p_1 \otimes p_2).$ is clear. Therefore, to show the result, it is enough to show that
$$N^{\perp} \subset \ker (p_1 \otimes p_2)^{\perp}.$$
Let $\sigma \colon F_1 \overset{h}{\otimes} F_2 \rightarrow \mathbb{C}$ be such that $\sigma_{|N}=0$. By Remark $\ref{DualH}$, there exist a Hilbert space $H$, $\alpha \colon X \rightarrow (H_c)^*$ linear and $\beta \colon Y \rightarrow H_c$ antilinear, $\alpha$ and $\beta$ completely bounded such that
$$\sigma(x, y) = \left\langle \alpha(x), \beta(y) \right\rangle, x\in F_1, y\in F_2.$$
Let $K = \overline{\alpha(F_1)}$ and denote by $P_K$ the orthogonal projection onto $K$. Then we have, for any $x$ and $y$,
$$\sigma(x,y) = \left\langle P_K \alpha(x), \beta(y) \right\rangle = \left\langle P_K \alpha(x), P_K \beta(y) \right\rangle.$$
Thus, by changing $\alpha$ into $P_K \alpha$ and $\beta$ into $P_K \beta$, we can assume that $\alpha$ has a dense range. Similarly, setting $L = \overline{\beta(F_2)}$ and considering $P_L$, we may assume that $\beta$ has a dense range.

By assumption, for any $e\in E_2$ and any $x\in E_1$, we have
$$0= \sigma(x,e) = \left\langle \alpha(x), \beta(e) \right\rangle.$$
This implies that $\beta_{|E_2}=0$. Similarly, we show that $\alpha_{|E_1}=0$. Thus, we can consider
$$\widehat{\alpha} \colon F_1/E_1 \rightarrow H \ \text{and} \ \ \widehat{\beta} \colon F_2/E_2 \rightarrow H$$
such that $\alpha = \widehat{\alpha} \circ p_1$ and $\beta = \widehat{\beta} \circ p_2$ and where $F_1/E_1$ and $F_2/E_2$ are equipped with their quotient structure. Now, define $\widehat{\sigma} \colon F_1/E_1 \overset{h}{\otimes} F_2/E_2 \rightarrow \mathbb{C}$ by
$$\widehat{\sigma}(s,t) = \left\langle \widehat{\alpha}(s), \widehat{\beta}(t) \right\rangle.$$
Then $\sigma = \widehat{\sigma} \circ (p_1 \otimes p_2)$, so that $\sigma \in \ker(p_1\otimes p_2)^{\perp}$.
\end{proof}

Finally, we recall the following \cite[Proposition 9.3.3]{Ruan} which will be important in the last section.

\begin{prop}\label{SimplOP}
Let $E$ be an operator space and let $\mathcal{H}$ and $\mathcal{K}$ be Hilbert spaces. For any $T\in CB(E, \mathcal{B}(\mathcal{H}, \mathcal{K}))$ we define a mapping
$\sigma_T \colon \mathcal{K}^* \otimes E \otimes \mathcal{H} \rightarrow \mathbb{C}$
by setting
$$\sigma_T(k^* \otimes e \otimes h) = \left\langle T(e)h, k \right\rangle.$$
Then, the mapping $T \mapsto \sigma_T$ induces a complete isometry
$$CB(E, \mathcal{B}(\mathcal{H}, \mathcal{K})) = \left( (\mathcal{K}_c)^* \overset{h}{\otimes} E \overset{h}{\otimes} \mathcal{H}_c \right)^*.$$
\end{prop}

\subsection{Schatten classes}\label{Schattenclasses}

Let $\mathcal{H}$ and $\mathcal{K}$ be separable Hilbert spaces. For any $1\leq p < +\infty$, let $\mathcal{S}^p(\mathcal{H}, \mathcal{K})$ be the space of compact operators $T \colon \mathcal{H} \rightarrow \mathcal{K}$ such that
$$
\|T\|_p \colon= \text{tr}(|T|^p)^{\frac{1}{p}} < \infty.
$$
$\|\cdot\|_p$ is a norm on $\mathcal{S}^p(\mathcal{H}, \mathcal{K})$ and $(\mathcal{S}^p(\mathcal{H}, \mathcal{K}), \|\cdot\|_p)$ is called the Schatten class of order $p$. When $p=\infty$, the space $\mathcal{S}^{\infty}(\mathcal{H}, \mathcal{K})$ will denote the space of compact operators equipped with the operator norm.\\
Recall that $\left( \mathcal{S}^1(\mathcal{H}, \mathcal{K}) \right)^* = \mathcal{B}(\mathcal{K}, \mathcal{H})$ and that for $1<p\leq +\infty$, $\left( \mathcal{S}^p(\mathcal{H}, \mathcal{K}) \right)^* = \mathcal{S}^{p'}(\mathcal{K}, \mathcal{H})$ where $p'$ is the conjugate exponent of $p$, for the duality pairing
$$
\left\langle S,T \right\rangle  = \text{tr}(ST),
$$
$S \in \mathcal{S}^p(\mathcal{H}, \mathcal{K})$ and $T\in \mathcal{S}^{p'}(\mathcal{K}, \mathcal{H})$.\\
Using the Haagerup tensor product introduced in Subsection $\ref{Haagerup}$, we have, by \cite[Proposition 9.3.4]{Ruan}, a complete isometry
\begin{equation}\label{traceOT}
(\mathcal{H}_c)^* \overset{h}{\otimes} \mathcal{K}_c = \mathcal{S}^1(\mathcal{H}, \mathcal{K}).
\end{equation}
where $\mathcal{S}^1(\mathcal{H}, \mathcal{K})$ is equipped with its operator space structure as the predual of $\mathcal{B}(\mathcal{K}, \mathcal{H})$.

\noindent Similarly, we have a complete isometry
\begin{equation}\label{compactOT}
\mathcal{K}_c \overset{h}{\otimes} (\mathcal{H}_c)^* = \mathcal{S}^{\infty}(\mathcal{H}, \mathcal{K}).
\end{equation}
Finally, if $(\Omega_1, \mu_1)$ and $(\Omega_2, \mu_2)$ are two $\sigma$-finite measure spaces, we will identify $L^2(\Omega_1 \times \Omega_2)$ with the space $\mathcal{S}^2(L^2(\Omega_1), L^2(\Omega_2))$ of Hilbert-Schmidt operators  as follows. If $K\in L^2(\Omega_1 \times \Omega_2)$, the operator
\begin{equation}\label{S2=L2not}
\begin{array}[t]{lccc}
X_K \colon & L^2(\Omega_1) & \longrightarrow & L^2(\Omega_2) \\
& f & \longmapsto & \displaystyle \int_{\Omega_1} K(t,\cdot)f(t) \mathrm{d}\mu_1(t)  \end{array}
\end{equation}
is a Hilbert-Schmidt operator and $\|X_J\|_2=\|J\|_{L^2}$. Moreover, any element of $\mathcal{S}^2(L^2(\Omega_1), L^2(\Omega_2))$ has this form.

\subsection{$L_{\sigma}^p$-spaces and duality}

Let $(\Omega, \mu)$ be a $\sigma$-finite measure space and let 
$F$ be a Banach space. For any $1\leq p \leq +\infty$,
we let $L^p(\Omega;F)$ denote the classical Bochner space of measurable functions
$f \colon\Omega\to F$.\\
Assume that $E$ is a separable Banach space. A function $f \colon \Omega \rightarrow E^*$ is said to be $w^*$-measurable if for all $ e\in E$, the function $t \in \Omega \mapsto \langle \phi(t), e \rangle$ is measurable. We denote by $L^p_{\sigma}(\Omega;E^*)$ the space of all $w^*$-measurable $f \colon \Omega \rightarrow E^*$ such that $\|f(\cdot)\| \in L^p(\Omega)$, after taking quotient by the functions which are equal to $0$ almost  everywhere. Equipped with the norm
$$
\| f\|_p = \| \|f(.)\| \|_{L^p(\Omega)},
$$
$(L^p_{\sigma}(\Omega; E^*), \|.\|_p)$ is a Banach space.\\
Let $1\leq p' \leq +\infty$ be the conjugate exponent of $p$. Then we have an isometric isomorphism
$$
L^p(\Omega;E)^* = L^{p'}_{\sigma}(\Omega; E^*)
$$
through the duality pairing 
\begin{equation*}
\langle f, g \rangle \colon= \int_{\Omega} 
\langle f(t), g(t) 
\rangle \,\text{d}\mu(t)\,.
\end{equation*}
See \cite[Section 4]{CLS} and the references therein for a proof of that result and more informations about $L_{\sigma}^p$-spaces.\\
Note that by \cite[Chapter IV]{Diestel}, the equality $L^p_\sigma(\Omega; E^*)=L^p(\Omega; E^*)$ is equivalent to $E^*$ having the Radon-Nikodym property. It is for instance the case for Hilbert spaces.

The important identification we will need in this paper is the following. For any $f\in L^\infty_\sigma(\Omega; E^*)$, define
\begin{equation}\label{Lpsigmaid}
u_f \colon \psi \in L^1(\Omega) \mapsto \left[ e\in E \mapsto \int_{\Omega} \left\langle f(t),e \right\rangle \psi(t) \,\text{d}t  \right] \in  E^*.
\end{equation}
Then $f \mapsto u_f$ yields an isometric identification (see \cite[Theorem 2.1.6]{DunPet})
\begin{equation}\label{DP}
L^\infty_\sigma(\Omega; E^*) = \mathcal{B}(L^1(\Omega),E^*).
\end{equation}
In particular, for a Hilbert space $\mathcal{H}$ we have the equality
\begin{equation}\label{DPHilbert}
L^{\infty}(\Omega; \mathcal{H}) = \mathcal{B}(L^1(\Omega), \mathcal{H}).
\end{equation}

\section{Multiple operator integrals}\label{MOI}

\subsection{Multiple operator integrals associated with operators}\label{Operators}

Let $\mathcal{H}$ be a separable Hilbert space and let $A$ be a (possibly unbounded) normal operator on $\mathcal{H}$. We denote by $\sigma(A)$ the spectrum of $A$ and 
by $E^A$ its spectral measure. A scalar-valued spectral measure for $A$ is a positive measure $\lambda_A$ on the Borel subsets of $\sigma(A)$ such that $\lambda_A$ and $E^A$ have the same sets of measure zero. Let $e$ be a separating vector of the von Neumann algebra $W^*(A)$ generated by $A$ (see \cite[Corollary 14.6]{Conway}). Then, by \cite[Proposition 15.3]{Conway}, the measure $\lambda_A$ defined by
$$
\lambda_A = \|E^A(.)e\|^2
$$
is a scalar-valued spectral measure for $A$.
We refer to \cite[Section 15]{Conway} and \cite[Section 2.1]{CLS} for more details.\\
For any bounded Borel function $f \colon \sigma(A) \to \mathbb{C}$, we define $f(A) \in \mathcal{B}(\mathcal{H})$ by
$$
f(A):=\int_{\sigma(A)} f(t) \ \text{d}E^A(t),
$$
and this operator only depends on the class of $f$ in $L^{\infty}(\lambda_A)$. According to \cite[Theorem 15.10]{Conway}, we obtain a $w^*$-continuous $*$-representation
$$
f \in L^{\infty}(\lambda_A) \mapsto f(A) \in \mathcal{B}(\mathcal{H}).
$$
Moreover, the space $L^{\infty}(\lambda_A)$ does not depend on the choice of the scalar-valued spectral measure. 

Let $n\in \mathbb{N}, n\geq 1$ and let $E_1, \ldots, E_n, E$ be Banach spaces. We denote by $\mathcal{B}_n(E_1 \times \cdots \times E_n, E)$ the space of $n$-linear continuous mappings from $E_1 \times \cdots \times E_n$ into $E$ equipped with the  norm
$$
\|T\|_{\mathcal{B}_n(E_1 \times \cdots \times E_n, E)} := \sup_{\|e_i\| \leq 1, 1\leq i \leq n} ~ \|T(e_1, \ldots, e_n)\|.
$$
When $E_1  = \cdots = E_n = E$, we will simply write $\mathcal{B}_n(E)$.

Let $n\in\mathbb{N}, n\geq 2$ and let $A_1, A_2, \ldots, A_n$ be normal operators in $\mathcal{H}$ with scalar-valued spectral measures $\lambda_{A_1}, \ldots, \lambda_{A_n}$. We let
\begin{equation*}
\Gamma^{A_1,A_2, \ldots, A_n} \colon L^{\infty}(\lambda_{A_1}) \otimes \cdots \otimes L^{\infty}(\lambda_{A_n}) \rightarrow \mathcal{B}_{n-1}(\mathcal{S}^2(\mathcal{H}))
\end{equation*}
to be the unique linear map such that for any $f_i \in L^{\infty}(\lambda_{A_i}), i=1, \ldots, n$ and for any $X_1, \ldots, X_{n-1} \in \mathcal{S}^2(\mathcal{H})$,
\begin{align*}
\left[\Gamma^{A_1,A_2, \ldots, A_n}(f_1\otimes\cdots\otimes f_n)\right]
& (X_1,\ldots, X_{n-1})\\ \nonumber
& =f_1(A_1)X_1f_2(A_2) \cdots f_{n-1}(A_{n-1})X_{n-1}f_n(A_n).
\end{align*}
We have a natural inclusion $L^{\infty}(\lambda_{A_1}) \otimes \cdots \otimes L^{\infty}(\lambda_{A_n}) \subset L^{\infty}\left(\prod_{i=1}^n
\lambda_{A_i}\right)$ which is $w^*$-dense. The following shows that $\Gamma^{A_1,A_2, \ldots, A_n}$ extends to $L^{\infty}\left(\prod_{i=1}^n
\lambda_{A_i}\right)$. It was proved in \cite[Theorem 4 and Proposition 5]{CLS}.

\begin{thm} $\Gamma^{A_1,A_2, \ldots, A_n}$ extends to a unique $w^*$-continuous isometry still denoted by
$$
\Gamma^{A_1,A_2, \ldots, A_n} \colon L^{\infty}\left(\prod_{i=1}^n
\lambda_{A_i}\right) \longrightarrow
\mathcal{B}_{n-1}(\mathcal{S}^2(\mathcal{H})).
$$
\end{thm}

\begin{defn}
For $\phi \in L^{\infty}\left(\prod_{i=1}^n\lambda_{A_i}\right)$, the transformation $\Gamma^{A_1,A_2, \ldots, A_n}(\phi)$ is called a multiple operator integral associated to $A_1, A_2, \ldots, A_n$ and $\phi$.
\end{defn}
\noindent The $w^*$-continuity of $\Gamma^{A_1,A_2, \ldots, A_n}$ means that if a net $(\phi_i)_{i\in I}$ in $L^{\infty}\left(\prod_{i=1}^n
\lambda_{A_i}\right)$ converges to $\phi \in L^{\infty}\left(\prod_{i=1}^n \lambda_{A_i}\right)$ in the $w^*$-topology, then for any $X_1, \ldots, X_{n-1} \in \mathcal{S}^2(\mathcal{H})$, the net
$$
\bigl(\left[\Gamma^{A_1,A_2, \ldots, A_n}(\phi_i)\right](X_1,\ldots, X_{n-1})\bigr)_{i\in I}
$$
converges to $\left[\Gamma^{A_1,A_2, \ldots, A_n}(\phi)\right](X_1,\ldots, X_{n-1})$ weakly in $\mathcal{S}^2(\mathcal{H})$. We refer to \cite[Section 3.1]{CLS} for more details. \\

\subsection{Continuous multilinear Schur multipliers}\label{Functions}

Let $n\in \mathbb{N}$. Let $(\Omega_1, \mu_1), \ldots, (\Omega_n, \mu_n)$ be $\sigma$-finite measure spaces, and let $\phi \in L^{\infty}(\Omega_1 \times\cdots \times \Omega_n)$. 
Let $\Omega = \Omega_2 \times \cdots \times \Omega_{n-1}$. For any $K_i \in L^2(\Omega_i\times \Omega_{i+1})$, $1\leq i \leq n-1$, we let $\Lambda(\phi)(K_1,\ldots, K_{n-1})$ to be the function
$$
(t_1,t_n) \mapsto \int_{\Omega} 
\phi(t_1,\ldots,t_n)K_1(t_1,t_2)\ldots K_{n-1}(t_{n-1},t_n) \,\text{d}\mu_2(t_2) \ldots \text{d}\mu_{n-1}(t_{n-1})
$$
By Cauchy-Schwarz inequality, $\Lambda(\phi)(K_1,\ldots, K_{n-1}) \in L^2(\Omega_1 \times \Omega_n)$ and
\begin{equation}\label{MOIFunctionsCont}
\| \Lambda(\phi)(K_1,\ldots, K_{n-1}) \|_2 \leq \|\phi\|_{\infty}\|K_1\|_2 \ldots \|K_{n_1}\|_2.
\end{equation}
Thus, $\Lambda(\phi)$ defines a bounded $(n-1)$-linear map
$$\Lambda(\phi) \colon L^2(\Omega_1 \times \Omega_2) \times L^2(\Omega_2\times \Omega_3) \times \cdots \times  L^2(\Omega_{n-1} \times \Omega_n) \longrightarrow L^2(\Omega_1 \times \Omega_n),$$
or, equivalently, by \eqref{S2=L2not} and the obvious equality $\mathcal{S}^2(L^2(\Omega_i), L^2(\Omega_j)) = \mathcal{S}^2(L^2(\Omega_j), L^2(\Omega_i)), 1\leq i,j \leq n,$ a bounded $(n-1)$-linear map
$$
\Lambda(\phi) \colon \mathcal{S}^2(L^2(\Omega_2), L^2(\Omega_1)) \times \cdots \times \mathcal{S}^2(L^2(\Omega_n), L^2(\Omega_{n-1})) \rightarrow \mathcal{S}^2(L^2(\Omega_n), L^2(\Omega_1)).
$$
For simplicity, write  $E_i = L^2(\Omega_i), 1\leq i  \leq n$. Then, the map $\Lambda \colon \phi \mapsto \Lambda(\phi)$ is a linear isometry
$$
\Lambda \colon L^{\infty}(\Omega_1 \times\cdots \times \Omega_n) \longrightarrow B_{n-1}(\mathcal{S}^2(E_2, E_1) \times \cdots \times \mathcal{S}^2(E_n, E_{n-1}), \mathcal{S}^2(E_n, E_1)).
$$
This follow e.g. from similar computations as those in the proof of \cite[Proposition 8]{CLS} or from \cite[Theorem 3.1]{JTT}.\\

Let $\mathcal{H}$ be a separable Hilbert space and let $A_1, \ldots, A_n$ be 
normal operators on $\mathcal{H}$. For any $1\leq i \leq n$, let $e_i \in \mathcal{H}$ be such that
$$
\lambda_{A_i}(\cdot)=\|E^{A^i}(\cdot)e_i\|^2.
$$
By \cite[Subsection 4.2]{CLS}, the linear mappings $\rho_i \colon L^2(\sigma(A_i), \lambda_{A_i}) \to \mathcal{H}$ defined for any measurable subset $F\subset \sigma(A_i)$ by
$$
\rho_i(\chi_F) = E^{A_i}(F)e_i
$$
extends uniquely to an isometry $\rho_i \colon L^2(\sigma(A_i), \lambda_{A_i}) \to \mathcal{H}$. Hence, denoting by $\mathcal{H}_i$ the range of $\rho_i$, we get that $\rho_i \colon L^2(\sigma(A_i), \lambda_{A_i}) \equiv \mathcal{H}_i$ is a unitary.

In the next result, we will consider the map $\Lambda$ introduced before and associated with the measure spaces $(\Omega_i,\mu_i)=(\sigma(A_i),\lambda_{A_i})$. We see any operator $T\in \mathcal{S}^2(\mathcal{H}_i, \mathcal{H}_j)$ as an element of $\mathcal{S}^2(\mathcal{H})$ by identifying $T$ with the matrix
$
\begin{pmatrix} T & 0 \\ 0 & 0\end{pmatrix}\ \in \,\mathcal{S}^2\bigl(\mathcal{H}_i\overset{2}\oplus \mathcal{H}_i^\perp,
\mathcal{H}_j\overset{2}\oplus \mathcal{H}_j^\perp\bigr).
$
The following makes the connection between the multiple operator integrals associated with operators and the map $\Lambda$ defined above. In particular, when one restricts the Hilbert space $\mathcal{H}$ to the subspaces $\mathcal{H}_i$, then the associated multiple operator integral coincides with $\Lambda$. It is the analogue of \cite[Proposition 9]{CLS} for $n$ operators. The proof is similar and we leave it to the reader.

\begin{prop}\label{Connection}
Let, for any $1\leq i \leq n-1, K_i \in \mathcal{S}^2(L^2(\lambda_{A_{i+1}}), L^2(\lambda_{A_i}))$ and set
$$
\widetilde{K}_i=\rho_i \circ K_i \circ \rho_{i+1}^{-1} \in \mathcal{S}^2(\mathcal{H}_{i+1}, \mathcal{H}_i).
$$
For any $\phi\in L^\infty(\lambda_{A_1} \times \cdots \times\lambda_{A_n})$, $\Gamma^{A_1, \ldots, A_n}(\phi)(\widetilde{K}_1, \ldots, \widetilde{K}_{n-1})$ belongs to $\mathcal{S}^2(\mathcal{H}_n,\mathcal{H}_1)$ and 
\begin{equation}\label{subspace}
\Lambda(\phi)(K_1, \ldots,K_{n-1}) = \rho_1^{-1} \circ 
\Gamma^{A_1, \ldots, A_n}(\phi)(\widetilde{K}_1, \ldots, \widetilde{K}_{n-1}) \circ \rho_n.
\end{equation}
\end{prop}

\section{Characterization of the complete boundedness of multiple operator integrals}\label{main}

Let $A_1, \ldots, A_n$ be $n$ normal operators on a separable Hilbert space $\mathcal{H}$ associated to scalar-valued spectral measures $\lambda_{A_1}, \ldots, \lambda_{A_n}$. For $\phi \in L^{\infty}(\lambda_{A_1} \times \cdots \times \lambda_{A_n})$, $\Gamma^{A_1, \ldots, A_n}(\phi)$ belongs to $\mathcal{B}_{n-1}(\mathcal{S}^2(\mathcal{H}))$, which is equivalent, by \cite[Section 3.1]{CLS}, to having a continuous mapping defined on the projective tensor product of $n-1$ copies $\mathcal{S}^2(\mathcal{H})$ and still denoted by
$$\Gamma^{A_1, \ldots, A_n}(\phi) \colon \mathcal{S}^2(\mathcal{H}) \overset{\wedge}{\otimes}\cdots \overset{\wedge}{\otimes} \mathcal{S}^2(\mathcal{H}) \rightarrow \mathcal{S}^2(\mathcal{H}).$$
We will make this identification for the rest of the paper.

The purpose of this section is to characterize the functions $\phi \in L^{\infty}(\lambda_{A_1} \times \cdots \times \lambda_{A_n})$ such that $\Gamma^{A_1, \ldots, A_n}(\phi)$ extends to a (completely) bounded map
$$\Gamma^{A_1, \ldots, A_n}(\phi) \colon \underbrace{\mathcal{S}^{\infty}(\mathcal{H}) \overset{h}{\otimes} \cdots \overset{h}{\otimes} \mathcal{S}^{\infty}(\mathcal{H})}_{n-1 \ \text{times}} \longrightarrow \mathcal{S}^{\infty}(\mathcal{H}).$$

We will also consider the continuous multilinear Schur multipliers $\Lambda(\phi)$. In \cite{JTT}, the authors studied and characterized the boundedness of continuous multilinear Schur multipliers
$$\mathcal{S}^{\infty}(L^2(\lambda_{A_{n-1}}), L^2(\lambda_{A_n})) \overset{h}{\otimes} \cdots  \overset{h}{\otimes} \mathcal{S}^{\infty}(L^2(\lambda_{A_1}),L^2(\lambda_{A_2})) \rightarrow \mathcal{S}^{\infty}(L^2(\lambda_{A_1}), L^2(\lambda_{A_n})).$$

They proved that we have such an extension if and only if $\phi$ has a certain factorization that will be given in the theorem below. They also proved that the boundedness for the Haagerup norm in this setting implies the complete boundedness.

The proof of Theorem $\ref{CBMOI}$ below includes another proof of \cite[Theorem 3.4]{JTT}. We show that for multiple operator integrals, boundedness and complete boundedness are also equivalent and that the same characterization holds.

\begin{thm}\label{CBMOI}
Let $n\in \mathbb{N}, n\geq 2$, let $A_1, \ldots, A_n$ be normal operators on a separable Hilbert space $\mathcal{H}$ and let $\phi \in L^{\infty}(\lambda_{A_1} \times \cdots \times \lambda_{A_n})$. For any $1\leq i \leq n$, let $E_i = L^2(\lambda_{A_i})$. The following are equivalent:
\begin{enumerate}
\item[(i)] $\Gamma^{A_1, \ldots, A_n}(\phi)$ extends to a bounded mapping
$$\Gamma^{A_1, \ldots, A_n}(\phi) \colon \mathcal{S}^{\infty}(\mathcal{H}) \overset{h}{\otimes} \cdots \overset{h}{\otimes} \mathcal{S}^{\infty}(\mathcal{H}) \rightarrow \mathcal{S}^{\infty}(\mathcal{H}).$$
\item[(ii)] $\Gamma^{A_1, \ldots, A_n}(\phi)$ extends to a completely bounded mapping
$$\Gamma^{A_1, \ldots, A_n}(\phi) \colon \mathcal{S}^{\infty}(\mathcal{H}) \overset{h}{\otimes} \cdots \overset{h}{\otimes} \mathcal{S}^{\infty}(\mathcal{H}) \rightarrow \mathcal{S}^{\infty}(\mathcal{H}).$$
\item[(iii)] $\Lambda(\phi)$ extends to a completely bounded mapping
$$\Lambda(\phi) \colon \mathcal{S}^{\infty}(E_2, E_1) \overset{h}{\otimes} \cdots  \overset{h}{\otimes} \mathcal{S}^{\infty}(E_n, E_{n-1}) \rightarrow \mathcal{S}^{\infty}(E_n, E_1).$$
\item[(iv)] There exist separable Hilbert spaces $H_1,  \ldots, H_{n-1}$, $a_1\in L^{\infty}(\lambda_{A_1} ; H_1), a_n \in L^{\infty}(\lambda_{A_n} ; H_{n-1})$
and
$a_i\in L_{\sigma}^{\infty}(\lambda_{A_i} ; \mathcal{B}(H_i, H_{i-1})), 2\leq i \leq n-1,$
such that
\begin{equation}\label{phifacto}
\phi(t_1,\ldots,t_n)= \left\langle a_1(t_1), [a_2(t_2)\ldots a_{n-1}(t_{n-1})](a_n(t_n)) \right\rangle
\end{equation}
for a.-e. $(t_1,\ldots,t_n) \in \sigma(A_1) \times \cdots \times \sigma(A_n).$
\end{enumerate}
In this case, 
\begin{equation*}
\left\|\Gamma^{A_1, \ldots, A_n}(\phi) \right\| = \left\|\Gamma^{A_1, \ldots, A_n}(\phi) \right\|_{\text{cb}} = \left\| \Lambda(\phi) \right\|_{\text{cb}}  = \inf \left\lbrace \|a_1\|_{\infty} \cdots \|a_n\|_{\infty} \ | \ \phi \ \text{as in} \ \eqref{phifacto} \right\rbrace.
\end{equation*}
\end{thm}

\smallskip

\begin{proof}

\noindent \underline{Proof of (i) $\Leftrightarrow$ (ii)}
\smallskip

Clearly (ii) $\Rightarrow$ (i) so we only prove (i) $\Rightarrow$ (ii). We keep the notation $\Gamma^{A_1, \ldots, A_n}(\phi)$ for the associated multilinear map defined on $\mathcal{S}^{\infty}(\mathcal{H}) \times \cdots \times \mathcal{S}^{\infty}(\mathcal{H})$.
Let $\mathcal{D} = W^*(A_1)'$ and $\mathcal{C} = W^*(A_n)'$ be the commutant of $W^*(A_1)$ and $W^*(A_n)$, respectively, where the von Neumann algebra $W^*(A)$ was defined in Section $\ref{Operators}$. Then $\Gamma^{A_1, \ldots, A_n}(\phi)$ is a multilinear $(\mathcal{D}, \mathcal{C})$-module map, that is, for any $d\in \mathcal{D}, c\in \mathcal{C}$, and any $X_1, \ldots, X_{n-1} \in \mathcal{S}^{\infty}(\mathcal{H})$,
\begin{equation}\label{bimod}
\left[\Gamma^{A_1, \ldots, A_n}(\phi)\right](dX_1, \ldots, X_{n-1}c) = d\left[\Gamma^{A_1, \ldots, A_n}(\phi)\right](X_1, \ldots, X_{n-1})c.
\end{equation}
By density, it is sufficient to check this equality when $X_i \in \mathcal{S}^2(\mathcal{H})$. But in this case, by linearity and $w^*$-continuity of $\Gamma^{A_1, \ldots, A_n}$, we can further assume that $\phi$ is an elementary tensor $\phi = f_1 \otimes \cdots \otimes f_n$, where $f_i \in L^{\infty}(\lambda_{A_i})$. Then, since $f_1(A_1) \in W^*(A_1)$ and $f_n(A_n) \in W^*(A_n)$ we have
\begin{align*}
& \left[\Gamma^{A_1, \ldots, A_n}(\phi)\right](dX_1, \ldots, X_{n-1}c) \\
& \ \ \ \ \ = f_1(A_1) d X_1 f_2(A_2) \ldots f_{n-1}(A_{n-1})X_{n-1}c f_n(A_n)\\
& \ \ \ \ \ = d f_1(A_1) X_1 f_2(A_2) \ldots f_{n-1}(A_{n-1})X_{n-1} f_n(A_n) c\\
& \ \ \ \ \ = d\left[\Gamma^{A_1, \ldots, A_n}(\phi)\right](X_1, \ldots, X_{n-1})c.
\end{align*}
Note that $W^*(A_1)$ has a separating vector and hence, by \cite[Proposition  14.3]{Conway}, this vector is cyclic for $\mathcal{D}$. Similarly, $\mathcal{C}$ has a cyclic vector. It remains to apply \cite[Lemma 3.3]{JTT} to obtain the complete boundedness of $\Gamma^{A_1, \ldots, A_n}(\phi)$ and the equality of the norms.

\smallskip
\noindent \underline{Proof of (ii) $\Rightarrow$ (iii)}
\smallskip

We use the same notations as in Subsection $\ref{Functions}$ where we introduced the subspaces $\mathcal{H}_i$ of $\mathcal{H}, 1\leq i \leq n$, with $\mathcal{H}_i \equiv L^2(\sigma(A_i), \lambda_{A_i})$. For any $1\leq i \leq n-1$,  $\mathcal{S}^{\infty}(\mathcal{H}_{i+1}, \mathcal{H}_i)$ is a closed subspace of $\mathcal{S}^{\infty}(\mathcal{H})$ and by injectivity of the Haagerup tensor product (see Proposition $\ref{Injectivity}$), we have a closed subspace
$$
\mathcal{S}^{\infty}(\mathcal{H}_2, \mathcal{H}_1) \overset{h}{\otimes} \cdots  \overset{h}{\otimes} \mathcal{S}^{\infty}(\mathcal{H}_n, \mathcal{H}_{n-1}) \subset \mathcal{S}^{\infty}(\mathcal{H}) \overset{h}{\otimes} \cdots  \overset{h}{\otimes} \mathcal{S}^{\infty}(\mathcal{H}).
$$
By Proposition $\ref{Connection}$, the restriction of $\Gamma^{A_1,\ldots,A_n}(\phi)$ to $\mathcal{S}^{\infty}(\mathcal{H}_2, \mathcal{H}_1) \overset{h}{\otimes} \cdots  \overset{h}{\otimes} \mathcal{S}^{\infty}(\mathcal{H}_n, \mathcal{H}_{n-1})$ is valued in $\mathcal{S}^{\infty}(\mathcal{H}_n, \mathcal{H}_1)$. Moreover, this restriction is completely bounded and by the same proposition, we obtain the inequality
$$\left\| \Lambda(\phi) \right\|_{\text{cb}} \leq \left\| \Gamma^{A_1,\ldots,A_n}(\phi) \right\|_{\text{cb}}.$$

\smallskip
\noindent \underline{Proof of (iii) $\Rightarrow$ (iv)}
\smallskip

In this part, the $L^1-$spaces will be equipped with their maximal operator space structure (Max) for which we refer to \cite[Chapter 3]{PisierCB}. If $(\Omega, \mu)$ is a measure space, the mapping
$(f,g) \in L^2(\Omega)^2 \mapsto fg \in L^1(\Omega)$
induces a quotient map
$$ f\otimes g \in L^2(\Omega) \overset{\wedge}{\otimes} L^2(\Omega) \mapsto fg \in L^1(\Omega).$$
We can identify $L^2(\Omega)$ with its conjugate space so that by $\eqref{S1}$ we get a quotient map
$$q \colon \mathcal{S}^1(L^2(\Omega)) \rightarrow L^1(\Omega)$$
which turns out to be a complete metric surjection.

Let $q_i \colon \mathcal{S}^1(L^2(\lambda_{A_i})) \rightarrow L^1(\lambda_{A_i}), i=1, \ldots, n$ be defined as above. Recall the notation $E_i = L^2(\lambda_{A_i})$. Using Proposition $\ref{Injectivity}$ together with the associativity of the Haagerup tensor product, we get a complete metric surjection
$$
Q = q_1 \otimes \cdots \otimes q_n \colon \mathcal{S}^1(E_1) \overset{h}{\otimes} \cdots \overset{h}{\otimes} \mathcal{S}^1(E_n) \rightarrow L^1(\lambda_{A_1}) \overset{h}{\otimes} \cdots \overset{h}{\otimes} L^1(\lambda_{A_n}).
$$
Let $N = \ker Q$ and let, for $1\leq i \leq n, N_i = \ker q_i$. For any $1\leq j \leq n$, let
$$
F_j = \mathcal{S}^1(E_1) \otimes \cdots \otimes \mathcal{S}^1(E_{j-1}) \otimes N_j  \otimes \mathcal{S}^1(E_j) \otimes \cdots \otimes \mathcal{S}^1(E_n).
$$
By Proposition $\ref{Injectivity} \ (iv)$, we obtain that
$$N = \overline{F_1  + F_2 + \cdots + F_n}.$$
Assume that $\Lambda(\phi)$ extends to a completely bounded mapping
$$\Lambda(\phi) \colon \mathcal{S}^{\infty}(E_2, E_1) \overset{h}{\otimes} \cdots  \overset{h}{\otimes} \mathcal{S}^{\infty}(E_n, E_{n-1}) \rightarrow \mathcal{S}^{\infty}(E_n, E_1).$$
Let $E = \mathcal{S}^{\infty}(E_2, E_1) \overset{h}{\otimes} \cdots  \overset{h}{\otimes} \mathcal{S}^{\infty}(E_n, E_{n-1})$. By Proposition $\ref{SimplOP}$, we have a complete isometry
$$CB(E, \mathcal{B}(E_n, E_1)) = \left( ((E_1)_c)^* \overset{h}{\otimes} E \overset{h}{\otimes} (E_n)_c \right)^*.$$
By $\eqref{compactOT}$ we have
$$E = (E_1)_c \overset{h}{\otimes} ((E_2)_c)^* \overset{h}{\otimes} (E_2)_c \overset{h}{\otimes} ((E_3)_c)^* \overset{h}{\otimes} \cdots \overset{h}{\otimes} (E_{n-1})_c \overset{h}{\otimes} ((E_n)_c)^*.$$
Thus, using $\eqref{traceOT}$ and the associativity of the Haagerup tensor product, we get that
$$CB(E, \mathcal{B}(E_n,  E_1)) = \left(\mathcal{S}^1(E_1) \overset{h}{\otimes} \cdots \overset{h}{\otimes} \mathcal{S}^1(E_n)\right)^*.$$
Let $u \colon \mathcal{S}^1(E_1) \overset{h}{\otimes} \cdots \overset{h}{\otimes} \mathcal{S}^1(E_n) \rightarrow \mathbb{C}$ induced by $\Lambda(\phi)$. For any $x_i \in \mathcal{S}^1(H_i), 1\leq i \leq n$, we have
$$u(x_1 \otimes \cdots \otimes x_n) = \int_{\Omega_1 \times \cdots \times \Omega_n} \phi(t_1,\ldots,t_n) [q_1(x_1)](t_1) \ldots [q_n(x_n)](t_n) ~ \text{d}\mu_1(t_1) \ldots \text{d}\mu_n(t_n).$$
To see this, it is enough to check it when the $x_i$ are rank one operators and in that case, one can use the identifications above. In particular, the latter implies that $u$ vanishes on $N = \ker Q$. Since $Q$ is a complete metric surjection, we get a mapping
$$v \colon L^1(\lambda_{A_1}) \overset{h}{\otimes} \cdots \overset{h}{\otimes} L^1(\lambda_{A_n}) \rightarrow \mathbb{C}$$
such that $u = v \circ Q$. An application of Theorem $\ref{DualityH}$ with suitable restrictions using the separability of the spaces $L^1(\lambda_{A_i})$ gives the existence of separable Hilbert spaces $H_1, \ldots, H_{n-1}$ and completely bounded maps
$$
\alpha_1 \colon L^1(\lambda_{A_1}) \rightarrow \mathcal{B}(H_1,\mathbb{C}) = (H_1)_c^*,
$$
$$
 \alpha_i \colon L^1(\lambda_{A_i}) \rightarrow \mathcal{B}(H_i, H_{i-1}), 2\leq i\leq n-1,
$$
$$
\alpha_n \colon L^1(\lambda_{A_n}) \rightarrow \mathcal{B}(\mathbb{C}, H_{n-1})= (H_{n-1})_c
$$
such that for any $f_j \in L^1(\lambda_{A_j}), 1\leq j \leq n$,
$$v(f_1 \otimes \cdots \otimes f_n) = \left\langle \alpha_1(f_1), [\alpha_2(f_2) \ldots \alpha_{n-1}(f_{n-1})](\alpha_n(f_n)) \right\rangle.$$
Since $L^1(\Omega_2)$ is equipped with the Max operator space structure, we have
$$CB(L^1(\lambda_{A_i}), \mathcal{B}(H_i, H_{i-1})) = \mathcal{B}(L^1(\lambda_{A_i}), \mathcal{B}(H_i, H_{i-1})).$$
Moreover, by $\eqref{DP}$, we have
$$ \mathcal{B}(L^1(\lambda_{A_i}), \mathcal{B}(H_i, H_{i-1})) = L^{\infty}_{\sigma}(\lambda_{A_i}; \mathcal{B}(H_i, H_{i-1})).$$
Thus, for any $2\leq i \leq n-1$, we associate to $\alpha_i$ an element $a_i \in L^{\infty}_{\sigma}(\lambda_{A_i}; \mathcal{B}(H_i, H_{i-1}))$. Similarly, we associate to $\alpha_1$ an element $a_1 \in L^{\infty}(\lambda_{A_1}; H_1)$ and to $\alpha_n$ an element  $a_n \in L^{\infty}(\lambda_{A_n}; H_{n-1})$. Using the identification $\eqref{Lpsigmaid}$, we obtain that
$$
\phi(t_1,\ldots,t_n)= \left\langle a_1(t_1), [a_2(t_2)\ldots a_{n-1}(t_{n-1})](a_n(t_n)) \right\rangle
$$
for a.-e. $(t_1,\ldots,t_n) \in \sigma(A_1) \times \cdots \times \sigma(A_n)$, and one can choose $a_1, \ldots, a_n$ such that we have the equality
$$
\left\| \Lambda(\phi) \right\|_{\text{cb}}  = \|a_1\|_{\infty} \cdots \|a_n\|_{\infty}.
$$

\smallskip
\noindent \underline{Proof of (iv) $\Rightarrow$ (ii)}
\smallskip

Assume that there exist separable Hilbert space $H_1 , \ldots, H_{n-1}$, $a_1\in L^{\infty}(\lambda_{A_1} ; H_1), a_i \in L_{\sigma}^{\infty}(\lambda_{A_i} ; \mathcal{B}(H_i, H_{i-1})), 2\leq i\leq n-1$ and $a_n \in L^{\infty}(\lambda_{A_n} ; H_{n-1})$ such that
$$
\phi(t_1,\ldots,t_n)= \left\langle a_1(t_1), [a_2(t_2)\ldots a_{n-1}(t_{n-1})](a_n(t_n)) \right\rangle
$$
for a.-e. $(t_1,\ldots,t_n) \in \sigma(A_1) \times \cdots \times \sigma(A_n)$. Let, for any $1\leq i\leq n-1$, $(\epsilon^i_n)_{n\geq 1}$ be a Hilbertian basis of $H_i$. Define, for $k,l \geq 1$,
$$a^1_k = \left\langle a_1, \epsilon^1_k \right\rangle, a^i_{kl} = \left\langle \epsilon^{i-1}_k, a_i \epsilon^i_l \right\rangle \ \ \text{and} \ \ a^n_l = \left\langle \epsilon^{n-1}_l, a_n \right\rangle.$$
Then $a^1_k \in L^{\infty}(\lambda_{A_1}), a^i_{kl} \in L^{\infty}(\lambda_{A_i}), 2\leq i \leq n-1$, and $a^n_l \in L^{\infty}(\lambda_{A_n})$. To see this, simply note that for $2\leq i \leq n-1$,
$$a^i_{kl} = {\rm tr} (a_i(\cdot) \circ (\epsilon^{i-1}_k \otimes \epsilon^i_l)).$$
For $N \geq 1$ and $1\leq i \leq n-1$, let $P^i_N$ be the orthogonal projection onto $\text{Span}(\epsilon^i_1, \ldots, \epsilon^i_N)$. Then, define
$$\phi_N = \left\langle P^1_N(a(t_1))), [a_2(t_2)P^2_N a_3(t_3)P^3_N \ldots a_{n-1}(t_{n-1})P^{n-1}_N](a_n(t_n)) \right\rangle.$$
It is clear that $(\phi_N)_{N\geq 1}$ is bounded in $L^{\infty}(\lambda_{A_1} \times \cdots \times \lambda_{A_n})$ and that $\phi_N \to \phi$ pointwise when $N \to \infty$. Therefore, by Dominated convergence theorem, we have that $\phi_N \to \phi$ for the $w^*-$topology. This implies, by $w^*-$ continuity of $\Gamma^{A_1,\ldots, A_n}$, that for any $X_j$ in $\mathcal{S}^2(\mathcal{H}), 1\leq j \leq n-1$,
$$\left[\Gamma^{A_1,  \ldots, A_n}(\phi_N)\right](X_1 \otimes \cdots \otimes X_{n-1}) \to \left[\Gamma^{A_1, \ldots, A_n}(\phi)\right](X_1 \otimes \cdots \otimes X_{n-1})$$
weakly in $\mathcal{S}^2(\mathcal{H})$.\\
Assume that $(\Gamma^{A_1, \ldots, A_n}(\phi_N))_N$ is uniformly bounded in $CB(\mathcal{S}^{\infty}(\mathcal{H}) \overset{h}{\otimes} \cdots \overset{h}{\otimes} \mathcal{S}^{\infty}(\mathcal{H}), \mathcal{S}^{\infty}(\mathcal{H})).$ Then, the above approximation property together with the density of $\mathcal{S}^2$ into $\mathcal{S}^{\infty}$ imply that $\Gamma^{A_1, \ldots, A_n}(\phi)$ is completely bounded as well with $\| \Gamma^{A_1, \ldots, A_n}(\phi) \|_{\text{cb}} \leq \sup_N \| \Gamma^{A_1, \ldots, A_n}(\phi_N) \|_{\text{cb}}$.

We will show now that for any $N\geq 1$, $\| \Gamma^{A_1, \ldots, A_n}(\phi_N) \|_{\text{cb}} \leq \|a_1\|_{\infty} \ldots \|a_n\|_{\infty}.$ For any $N\geq 1$ and a.-e. $(t_1,\ldots,t_n) \in \sigma(A_1) \times \cdots \times \sigma(A_n)$, we have
$$\phi_N(t_1, \ldots, t_n) = \sum_{k_1, \ldots, k_{n-1} = 1}^N a^1_{k_1}(t_1) a^2_{k_1 k_2}(t_2) \ldots a^{n-1}_{k_{n-2} k_{n-1}}(t_{n-1}) a^n_{k_n}(t_n),$$
so that for any $X_1, \ldots,  X_{n-1} \in \mathcal{S}^2(\mathcal{H})$,
\begin{align*}
& \left[\Gamma^{A_1, \ldots, A_n}(\phi_N)\right](X_1 \otimes \cdots \otimes X_{n-1}) \\
& = \sum_{k_1, \ldots, k_{n-1} = 1}^N a^1_{k_1}(A_1) X_1 a^2_{k_1 k_2}(A_2) X_2 \ldots X_{n-2} a^{n-1}_{k_{n-2} k_{n-1}}(A_{n-1}) X_{n-1} a^n_{k_n}(A_n).
\end{align*}
Note that the latter can be written as
$$\left[\Gamma^{A_1, \ldots, A_n}(\phi_N)\right](X_1 \otimes \cdots \otimes X_{n-1})= A^1_N (X_1 \otimes I_N) A^2_N (X_2 \otimes I_N) \cdots (X_{n-1} \otimes I_N) A^n_N,$$
where
$$A^1_N = [a^1_1(A_1) \ a^1_2(A_1) \ldots a^1_N(A_1)] \colon \ell_2^N(\mathcal{H}) \rightarrow \mathcal{H},$$

$$A^i_N = [a^i_{kl}(A_i)]_{\begin{subarray}{l}
1 \leq k \leq N\\ 1 \leq l \leq N
\end{subarray}} \colon \ell_2^N(\mathcal{H}) \rightarrow \ell_2^N(\mathcal{H}), \ 2 \leq i \leq n-1$$
and
$$A^n_N = [a^n_1(A_n) \ a^n_2(A_n) \ldots A^n_N(A_n)]^t \colon \mathcal{H} \rightarrow \ell_2^N(\mathcal{H}).$$
The notation $X \otimes I_N$ stands for the element of $\mathcal{B}(\ell_2^N(\mathcal{H}))$ whose matrix is the $N\times N$ diagonal matrix $\text{diag}(X, \ldots, X)$.

For any $N\geq 1$ and any $1\leq i \leq n$, let $\pi_N$ and $\pi_i$ be the $*-$representations defined by
\begin{equation*}
\begin{array}[t]{lrcl}
\pi_N \colon &\mathcal{B}(\mathcal{H}) & \longrightarrow & \mathcal{B}(\ell_2^N(\mathcal{H}))\\
&  X & \longmapsto & X\otimes I_N \end{array} \ \ \text{and} \ \
\begin{array}[t]{lrcl}
\pi_{A_i} \colon &L^{\infty}(\lambda_{A_i}) & \longrightarrow & \mathcal{B}(\mathcal{H})\\
&  f & \longmapsto & f(A_i) \end{array}.
\end{equation*}
By \cite[Proposition 1.5]{PisierCB}, $\pi_N$ and $\pi_{A_i}$ are completely bounded with cb-norm less than $1$. Note that the element
$[a^i_{kl}]_{1 \leq k,l \leq N} \in M_N(L^{\infty}(\lambda_B))$
has a norm less than $\|a_i\|_{\infty}$. Thus, the latter implies that
$A^i_N = [\pi_{A_i}(a^i_{kl})]_{1 \leq k,l \leq N}$
has an operator norm less than $\|a_i\|_{\infty}$. Similarly (using column and row matrices), we show that $A^1_N$ and $A^n_N$ have a norm less than $\|a_1\|_{\infty}$ and $\|a_n\|_{\infty}$, respectively. Finally, write
$$\left[\Gamma^{A_1, \ldots, A_n}(\phi_N)\right](X_1 \otimes \cdots \otimes X_{n-1}) = \sigma^1_N(X_1) \sigma^2_N(X_2) \ldots \sigma^{n-1}_N(X_{n-1}),$$
where for any $1 \leq i \leq n-2$, $\sigma^i_N(X_1) = A^i_N \pi_N(X_i)$ and $\sigma_{n-1}^N(X_{n-1}) = A^{n-1}_N \pi_N(X_{n-1}) A^n_N$. By the easy part of Wittstock theorem (see e.g. \cite[Theorem 1.6]{PisierCB}), $\sigma^i_N$ and $\sigma_{n-1}^N$ are completely bounded with cb-norm less than $\|a_i\|_{\infty}$ and $\|a_{n-1}\|_{\infty} \|a_n\|_{\infty}$, respectively. Hence, by Theorem $\ref{DualityH}$, we get that $\Gamma^{A_1, \ldots, A_n}(\phi_N)$ is completely bounded with cb-norm less than $\|a_1\|_{\infty} \ldots \|a_n\|_{\infty}$.
\end{proof}

\vskip 0.5cm
\noindent
{\bf Acknowledgements.} The author is supported by NSFC (11801573).

\end{document}